\documentclass[11pt]{amsart}
\usepackage{geometry}                
\usepackage{graphicx}
\usepackage{amssymb}
\usepackage{amsthm}
\usepackage{mathrsfs}
\usepackage{epstopdf}
\usepackage{MnSymbol}
\usepackage{stmaryrd}
\usepackage{verbatim}
\usepackage[all]{xy}
\title{Twistor Spaces of Hypercomplex Manifolds are Balanced}
\author{Artour Tomberg}
\thanks{The author would like to thank his research supervisor Misha Verbitsky for suggesting the problem and generously helping in the preparation of this manuscript. The article was prepared within the framework of the Academic Fund Program at the National Research University Higher School of Economics (HSE) in 2015 (grant 15-05-0033)  and supported within the framework
of a subsidy granted to the HSE by the Government of the Russian Federation for the implementation of the Global Competitiveness Program.}
\address{Faculty of Mathematics, National Research University Higher School of Economics, 
7 Vavilova Str., Moscow, Russia, 117312}
\address{Department of Mathematics and Statistics, McGill University, 805 Sherbrooke Street West, Montreal, Quebec,  Canada, H3A 0B9}
\email{artour@tomberg.com}

\DeclareMathOperator{\End}{End}
\DeclareMathOperator{\Hom}{Hom}

\DeclareMathOperator{\Imag}{Im}
\DeclareMathOperator{\id}{id}

\begin{document}

\begin{abstract}
A hypercomplex structure on a differentiable manifold consists of three integrable almost complex structures that satisfy quaternionic relations. If, in addition, there exists a metric on the manifold which is Hermitian with respect to the three structures, and such that the corresponding Hermitian forms are closed, the manifold is said to be hyperk\"ahler. In the paper ``Non-Hermitian Yang-Mills connections'' \cite{kaled-verbit}, Kaledin and Verbitsky proved that the twistor space of a hyperk\"ahler manifold admits a balanced metric; these were first studied in the article ``On the existence of special metrics in complex geometry'' \cite{michel} by Michelsohn. In the present article, we review the proof of this result and then generalize it and show that twistor spaces of general compact hypercomplex manifolds are balanced.
\vskip 0.25 cm
\noindent \textbf{Keywords.} Hypercomplex geometry; Twistor theory; Balanced manifold.
\end{abstract}

\maketitle

\tableofcontents

\section{Introduction}
The rapid progress in K\"ahler geometry in the middle of the XXth century opened two natural directions of research in complex geometry: on the one hand, a quest for a suitable quaternionic analogue of K\"ahler manifolds, and, on the other hand, the study of various generalizations of the K\"ahler condition $d\omega = 0$ on Hermitian metrics. The former question has been settled by Calabi in \cite{calabi} with the introduction of hyperk\"ahler manifolds. These are smooth manifolds $M$ with three integrable almost complex structures $I, J, K : TM \to TM$ satisfying the quaternionic relations $I^2 = J^2 = K^2 = -1$, $IJ = -JI = K$, together with a metric $g$ that preserves the three complex structures and such that the form $\Omega_I = \omega_J + \sqrt{-1}\omega_K$ is closed. Hyperk\"ahler manifolds appeared in the much earlier work of Berger on the classification of irreducible holonomy groups on Riemannian manifolds \cite{berger}, where they correspond to the holonomy group $Sp(n)$. Yau's proof of the Calabi conjecture \cite{yau} provides a wealth of examples of compact hyperk\"ahler manifolds by showing that these are equivalent to holomorphic symplectic manifolds. If we forget about the metric $g$ and look only at the complex structures on $M$, the resulting structure is called hypercomplex. Hypercomplex manifolds were first studied by Boyer in \cite{boyer}, where he gave their complete classification in quaternionic dimension 1. In addition to hyperk\"ahler metrics on hypercomplex manifolds, one can study their various generalizations such as HKT metrics which are characterized by the condition $\partial \Omega_I = 0$, where the decomposition $d = \partial + \bar{\partial}$ is with respect to the structure $I$; HKT metrics were first introduced in \cite{howe-pap}.

For the generalizations of the K\"ahler condition $d\omega = 0$, there is the famous classification result of Gray and Hervella. In their paper \cite{gray-hervella}, they used representation theory to define and study sixteen classes of almost Hermitian manifolds, each of which presents a generalization of the K\"ahler condition on metrics. A particularly useful class is that of balanced metrics, characterized by the condition $d^* \omega = 0$, where $d^*$ is the dual operator of $d$ with respect to the given metric. The terminology comes from the paper \cite{michel} of Michelsohn where these metrics were first studied in depth; in the terminology of Gray and Hervella, these are called special Hermitian metrics and correspond to the class $\mathcal{W}_3$. Balanced metrics form a strictly greater class than K\"ahler metrics and are dual to them in a sense made precise in \cite{michel}. An example illustrating the importance of balanced metrics comes from the paper \cite{verbit2} of Verbitsky, where he showed that balanced HKT metrics play the role of Calabi-Yau metrics for the quaternionic Calabi conjecture (yet to be proven). Another area where balanced metrics come to the fore is the theory of stability of vector bundles and the Kobayashi-Hitchin correspondence (see, e.g., \cite{lubke-teleman} for reference). The notion of stability was first introduced by Mumford in a purely algebro-geometric setting in \cite{mumford} for projective varieties and then generalized to K\"ahler manifolds and then to general Hermitian manifolds. Stable vector bundles are important because they form moduli spaces with meaningful structure. The Kobayashi-Hitchin correspondence relates these (essentially algebro-geometric) moduli spaces of stable vector bundles to the moduli spaces of Einstein-Hermitian vector bundles, a purely differential-geometric notion introduced by Kobayashi in \cite{kobayashi}. This correspondence was conjectured independently by Kobayashi and Hitchin, and then gradually proved in increasing generality: first for algebraic curves \cite{donaldson}, surfaces \cite{donaldson2} and manifolds \cite{donaldson3} by Donaldson, then for K\"ahler manifolds by Uhlenbeck and Yau \cite{uhlenbeck-yau,uhlenbeck-yau2}, and finally for general Hermitian manifolds by Li and Yau \cite{li-yau}. Although the result of Li and Yau establishes the correspondence for a general complex manifold $M$, the theory becomes more complicated compared to the K\"ahler case, since the notion of degree of a vector bundle, needed to define stability, is no longer a topological invariant of the vector bundle, as in the case of K\"ahler manifolds, but only a holomorphic one. However, as shown in \cite{li-yau}, if $M$ is balanced, the degree still turns out to be a topological invariant, thus making the theory in the case of balanced manifolds much simpler than in the general case.

This property of balanced manifolds is extensively used by Kaledin and Verbitsky in \cite{kaled-verbit}. Among other things, they use the twistor formalism to establish a correspondence between non-Hermitian Yang-Mills connections over a hyperk\"ahler manifold $M$ and holomorphic bundles over its twistor space $Z$, which essentially encodes the quaternionic structure of $M$ in purely holomorphic data. The twistor space $Z$ is never K\"ahler, but it is balanced, as they show in Section 4.4 of \cite{kaled-verbit}, essentially generalizing a result from the original paper \cite{michel} of Michelsohn stating that twistor spaces of self-dual Riemannian 4-manifolds are balanced. They then use the result of Li and Yau to study the moduli space of stable bundles over the twistor space $Z$. It is the goal of our exposition to present the argument of Kaledin and Verbitsky on the balancedness of the twistor space $Z$ of a hyperk\"ahler manifold $M$, and then extend it to the case of an arbitrary (compact) hypercomplex manifold $M$.

\section{Balanced manifolds} Our first goal is to give the definition of balanced metrics on manifolds. We start with some preliminaries from differential geometry.

Let $M$ denote a (real) $C^\infty$-manifold and $E \to M$ a (real) $C^\infty$-vector bundle over $M$. Recall that a \emph{connection} on $E$ is an $\mathbb{R}$-linear operator $\nabla : \Gamma(E) \to \Gamma(\Lambda^1 M \otimes E)$ satisfying the Leibniz rule:
\[
\nabla(fs) = df \otimes s + f \nabla s \ \ \forall f \in C^\infty M, s \in \Gamma(E).
\]
Given a vector field $X \in \Gamma(TM)$, we denote by $\nabla_X s \in \Gamma(E)$ the usual pairing of  $X$ with $\nabla s \in \Gamma\left(\Lambda^1 M \otimes E\right)$. Associated to a connection $\nabla$ is its \emph{curvature} $R^\nabla : \Lambda^2 (TM) \to \End(TM)$ defined by
\[
R^\nabla(X, Y) := \nabla_X \nabla_Y - \nabla_Y \nabla_X - \nabla_{\left[X, Y\right]} \ \ \forall X, Y \in \Gamma\left(TM\right).
\]
In the special case of $E = TM$ being the tangent bundle, we can also define the \emph{torsion} $T^\nabla : \Lambda^2 (TM) \to TM$ of the connection by
\[
T^\nabla(X, Y) := \nabla_X Y - \nabla_Y X - \left[X, Y\right] \ \ \forall X, Y \in \Gamma\left(TM\right).
\]
In fact, it's easy to verify that both $R^\nabla$ and $T^\nabla$ are $C^\infty$-linear operators, hence we can think of them as tensors: $R^\nabla \in \Gamma(\Lambda^2 M \otimes \End(TM))$, $T^\nabla \in \Gamma(\Lambda^2 M \otimes TM)$. If $R^\nabla = 0$, the connection is said to be \emph{flat}, while if $T^\nabla = 0$, it is called \emph{torsion-free}.

Observe that a connection $\nabla : \Gamma(E) \to \Gamma(\Lambda^1 M \otimes E)$ on $E$ induces a canonical connection on the dual bundle $E^* = \Hom_{\mathbb{R}}(E, \mathbb{R})$, also denoted by $\nabla$, and defined by
\[
\left< \nabla \eta, s \right> + \left< \eta, \nabla s \right> = d\left(\left< \eta, s \right>\right) \ \ \forall \eta \in \Gamma(E^*), \, s \in \Gamma(E),
\]
where we denote by $\left< \, ,\right>$ the pairing of $E^*$ with $E$. Given connections $\nabla^E, \nabla^F$ on vector bundles $E, F$, we can consider the induced connections $\nabla^{E \oplus F}$, $\nabla^{E \otimes F}$ on $E \oplus F$, $E \otimes F$ defined by
\[
\nabla^{E \oplus F}(s \oplus t) := \left(\nabla^E s\right) \oplus \left(\nabla^F t\right) \ \ \forall s \in \Gamma(E), \, t \in \Gamma(F).
\]
\[
\nabla^{E \otimes F}(s \otimes t) := \left(\nabla^E s\right) \otimes t + s \otimes \left(\nabla^F t\right) \ \ \forall s \in \Gamma(E), \, t \in \Gamma(F).
\]
Thus, starting with a single connection $\nabla$ on $E$, we can form induced connections on all tensor products $\left(E^*\right)^{\otimes r} \otimes E^{\otimes q}$. Moreover, it's not hard to see that the subspaces of symmetric and antisymmetric tensors are invariant under these connections. In what follows, all these induced connections on tensor powers of $E$ will be denoted by the same symbol $\nabla$, and when $\nabla s = 0$ for some tensor $s$, we will say that the connection \emph{preserves} $s$.

We now specialize to the case when $M$ is a complex manifold and $E \to M$ is a complex vector bundle. Since $E$ is in particular a real vector bundle, we can have connections on $E$ defined as above, but this time we can single out those that are $\mathbb{C}$-linear as operators $\Gamma(E) \to \Gamma(\Lambda^1 M \otimes E)$; these are precisely the connections which preserve the operator $I : E \to E$, $I^2 = -1$, of multiplication by the imaginary unit in $E$ viewed as a complex vector bundle. In addition to the induced connections described in the previous paragraph, a $\mathbb{C}$-linear connection $\nabla$ on $E$ induces $\mathbb{C}$-linear connections on the complex dual $E^* = \Hom_{\mathbb{C}}(E, \mathbb{C})$ and the conjugate bundle $\bar{E}$. For the special case that $E = TM$ is the tangent bundle, the operator $I : TM \to TM$ above is called the almost complex structure of $M$. It is a well-known result that the condition of $M$ being a complex manifold is equivalent to the \emph{integrability} of $I$, i.e. the existence of a torsion-free connection $\nabla$ that preserves $I$ \cite{newlander}.

There is a canonical eigenvalue decomposition of the operator $I$ on the complexified tangent bundle $T_{\mathbb{C}}M = TM \otimes_{\mathbb{R}} \mathbb{C} = T^{1,0}M \oplus T^{0,1}M$, where
\[
T^{1,0}M = \left\{ v \in T_{\mathbb{C}}M : Iv = \sqrt{-1}v\right\} = \left\{X - \sqrt{-1}IX : X \in TM\right\},
\]
\[
T^{0,1}M = \left\{ v \in T_{\mathbb{C}}M : Iv = -\sqrt{-1}v\right\} = \left\{X + \sqrt{-1}IX : X \in TM\right\}.
\]
Observe that $TM \cong T^{1,0}M$ as complex bundles, while $T^{0,1}M$ is the dual of $T^{1,0}M$. We can also define the induced operator $I : T^* M \to T^* M$ on the cotangent bundle by putting $I\Omega(X) := -\Omega(IX)$, and more generally on $\Lambda^n M$ by $I\left(\Omega_1 \wedge \ldots \wedge \Omega_n\right) = \left(I\Omega_1\right) \wedge \ldots \wedge \left(I\Omega_n\right)$. There is a similar decomposition $T^*_{\mathbb{C}}M = T^* M \otimes_{\mathbb{R}} \mathbb{C} = \left(T^*\right)^{1,0}M \oplus \left(T^*\right)^{0,1}M$, where
\[
\left(T^*\right)^{1,0}M = \left\{\omega \in T^*_{\mathbb{C}}M : \omega(v) = 0 \ \forall v \in T^{0,1}M\right\} = \left\{\Omega + \sqrt{-1}I\Omega : \Omega \in T^*M\right\},
\]
\[
\left(T^*\right)^{0,1}M = \left\{\omega \in T^*_{\mathbb{C}}M : \omega(v) = 0 \ \forall v \in T^{1,0}M\right\} = \left\{\Omega - \sqrt{-1}I\Omega : \Omega \in T^*M\right\}.
\]
The higher differential forms on $M$ can then be decomposed as
\[
\Lambda^k_{\mathbb{C}}M = \Lambda^k M \otimes_{\mathbb{R}} \mathbb{C} = \Lambda^{k,0} M \oplus \Lambda^{k-1,1} M \oplus \ldots \Lambda^{1,k-1} M \oplus \Lambda^{0,k} M,
\]
where
\[
\Lambda^{p,q} M \cong \Lambda^p \left(\left(T^*\right)^{1,0}M\right) \otimes \Lambda^q \left(\left(T^*\right)^{0,1}M\right).
\]
The (real) exterior derivative operator $d : \Lambda^k M \to \Lambda^{k+1} M$ can be extended by $\mathbb{C}$-linearity to $\Lambda^k_{\mathbb{C}} M$, and on the spaces $\Lambda^{p,q}M$ as above, it decomposes as $d  = \partial + \bar{\partial}$, where
\[
\partial : \Lambda^{p,q}M \longrightarrow \Lambda^{p+1,q}M, \ \bar{\partial} : \Lambda^{p,q}M \longrightarrow \Lambda^{p,q+1}M.
\]
We also introduce the differential operator $d^c = \sqrt{-1}\left(\bar{\partial} - \partial\right)$ for convenience. Observe that $d^c$ is a real operator like $d$, i.e. it takes real forms to real forms.

Define a \emph{Hermitian metric} on the complex manifold $M$ of $\dim_{\mathbb{C}}M = n$ to be a Riemannian metric $g$ on the tangent bundle $TM$ satisfying
\[
g(IX, IY) = g(X, Y) \ \ \forall X, Y \in \Gamma(TM),
\]
where $I$ is the almost complex structure. Hermitian metrics always exist, in fact, starting with an arbitrary Riemannian metric $g_0$ on $TM$, we can define
\[
g(X, Y) := g_0(X, Y) + g_0(IX, IY) \ \ \forall X, Y \in \Gamma(TM),
\]
and this is clearly Hermitian. Observe that, as a Riemannian metric, $g$ induces a (real) bundle isomorphism $TM \cong \Lambda^1 M$. We can also extend $g$ by $\mathbb{C}$-linearity to the complexified tangent bundle $T_{\mathbb{C}}M$, where it induces an isomorphism of complex vector bundles $T^{1,0} M \cong \Lambda^{0,1} M$. Given an arbitrary Hermitian manifold $(M, I, g)$, there are two canonical connections on its tangent bundle $TM$. 
\begin{enumerate}
\item[1.] The Levi-Civita connection $\nabla^{LC}$ is the unique $\mathbb{R}$-linear connection which preserves the metric tensor $g$ and whose torsion is zero.
\item[2.] The Chern connection $\nabla^{Ch}$ is the unique $\mathbb{C}$-linear connection which preserves $g$ and whose torsion tensor lies in $\Lambda^{2,0} M \otimes_{\mathbb{C}} T^{1,0}M \subseteq \left( \Lambda^2M \otimes_{\mathbb{R}} \mathbb{C}\right)  \otimes_{\mathbb{C}} T^{1,0}M \cong \Lambda^2 M \otimes_{\mathbb{R}} TM$, where $T^{1,0}M$ and $(TM, I)$ are identified as complex vector bundles.
\end{enumerate}

Associated to each Hermitian metric is its \emph{Hermitian form} $\omega \in \Lambda^2 M$ given by
\[
\omega(X, Y) := g(IX, Y) \ \ \forall X, Y \in \Gamma(TM).
\]
It's easy to verify that $\omega$ is a non-degenerate real $(1,1)$-form which satisfies the \emph{strict positivity} property:
\[
\omega(X, IX) > 0 \ \ \forall X \ne 0 \in \Gamma(TM).
\]
If the Hermitian form $\omega$ is closed, $M$ is called a \emph{K\"ahler manifold}; in case the weaker condition $d(\omega^{n-1}) = 0$ is satisfied, it is called a \emph{balanced manifold}. It can be shown that the K\"ahler condition is equivalent to $\nabla^{LC} I = 0$, and also to the vanishing of the torsion tensor $T^{Ch}$ of the Chern connection $\nabla^{Ch}$. On the other hand, observe that $T^{Ch}$ lies in
\[
\Lambda^{2,0} M \otimes_{\mathbb{C}} T^{1,0}M \subseteq \Lambda^{1,0}M \otimes_{\mathbb{C}} \left( \Lambda^{1,0}M \otimes_{\mathbb{C}} T^{1,0}M\right) \cong \Lambda^{1,0}M \otimes_{\mathbb{C}} \End_{\mathbb{C}}(T^{1,0}M),
\]
and it can be shown (see Theorem 1.6 in \cite{michel}) that the vanishing of the $(1,0)$-form obtained by taking the complex trace pairing on $\End_{\mathbb{C}}(T^{1,0}M)$ of the tensor $T^{Ch}$ is equivalent to the balancedness condition on the metric.

Observe that in dimension $\dim_{\mathbb{C}}M = 2$, the balancedness condition is equivalent to the K\"ahler condition, since in this special case $\omega^{n-1} = \omega$. In general dimension, however, the condition of being K\"ahler is stronger than that of being balanced. Examples of balanced non-K\"ahler manifolds are twistor spaces $Z$ of certain self-dual Riemannian 4-manifolds $M$. These are 3-dimensional complex manifolds which encode the conformal structure of $M$. They are always balanced (see \cite{michel}, Section 6), but, as shown by N. Hitchin in \cite{hitchin}, the twistor space $Z$ is K\"ahler only if $M = S^4$ or $\mathbb{CP}^2$. In the next two sections, we will extend the balancedness result to twistor spaces of hyperk\"ahler manifolds (following \cite{kaled-verbit}) and general compact hypercomplex manifolds.

We end this section with two lemmas of an essentially linear-algebraic nature that will be useful for us in Section \ref{sec:hypercomplex}. Recall that a real (1,1)-form $\eta$ on a complex manifold $(M, I)$ is strictly positive if it satisfies the condition $\eta(X, IX) > 0$ for all nonzero $X \in \Gamma(TM)$. Similarly, we say that a real $(n-1, n-1)$-form $\eta$ is strictly positive if for any nonzero $\alpha \in \Lambda^1 M$ we have that $\eta \wedge \alpha \wedge I\alpha$ is a strictly positive multiple of (any) volume form on $M$ compatible with the orientation determined by the complex structure. There is an intimate relationship between closed strictly positive $(n-1,n-1)$-forms on $M$ and balanced metrics.

\theoremstyle{plain}
\newtheorem{positive}{Lemma}
\begin{positive} \label{thm:positive}
Let $(M, I, g)$ be a Hermitian manifold of $\dim_{\mathbb{C}}M = n$. The existence of a closed strictly positive $(n-1, n-1)$-form on $M$ is equivalent to the balancedness of $M$, not necessarily with respect to the given metric.
\end{positive}

\begin{proof} (Cf. \cite{michel}, pp. 279-280)
Let $\eta \in \Lambda^{n-1,n-1}M$ be a closed strictly positive form. The Riemannian volume form $\Omega \in \Lambda^{2n}M$ induces an isomorphism of bundles $\Lambda^{n-1,n-1}M \cong \Lambda^{1,1}TM \cong T^{1,0}M \otimes T^{0,1}M$, whereas the metric $g$ gives an isomorphism $\Lambda^{1,1}TM \cong \Lambda^{1,1}M$. Under these identifications, $\eta$ can be thought of as a strictly positive $(1,1)$-form on $M$. By basic linear algebra, there exists a local orthonormal frame $\left\{e_1, Ie_1, \ldots, e_n, Ie_n\right\}$ of $TM$, such that $\eta \in \Lambda^{1,1}M$ can be expressed as
\[
\eta = \sum_{i=1}^n a_i \, e_i \wedge Ie_i,
\]
where we think of $e_i$ as elements of $ \Lambda^1 M \cong TM$ and all $a_i > 0$. Since $\Omega = e_1 \wedge Ie_1 \wedge \ldots \wedge e_n \wedge Ie_n$, we have that, as element of $\Lambda^{n-1,n-1}M$, $\eta$ can be expressed in terms of this basis as
\[
\eta = \sum_{i=1}^n a_i \, e_1 \wedge Ie_1 \wedge \ldots \wedge \widehat{e_i \wedge Ie_i} \wedge \ldots \wedge e_n \wedge Ie_n.
\]
We are now looking for a strictly positive form $\omega \in \Lambda^{1,1}M$ such that $\omega^{n-1} = \eta$. If we can establish the existence of such a form, our proof will be finished, since the condition $d\left(\omega^{n-1}\right) = 0$ will imply that the Hermitian metric on $M$ induced by $\omega$ is balanced. If we write
\[
\omega = \sum_{i=1}^n b_i \, e_i \wedge Ie_i,
\]
we then have
\[
\omega^{n-1} = \sum_{i=1}^n (n-1)! \, b_1 \ldots \widehat{b_i} \ldots b_n \, e_1 \wedge Ie_1 \wedge \ldots \wedge \widehat{e_i \wedge Ie_i} \wedge \ldots \wedge e_n \wedge Ie_n.
\]
If $\omega^{n-1} = \eta$, observe that
\[
\frac{a_i}{a_j} = \frac{(n-1)! \, b_1 \ldots \widehat{b_i} \ldots b_n}{(n-1)! \, b_1 \ldots \widehat{b_j} \ldots b_n} = \frac{b_j}{b_i}.
\]
Writing
\[
a_1 = (n-1)! \, b_2 \ldots b_n = (n-1)! \, \frac{b_2}{b_1} \ldots \frac{b_n}{b_1} \cdot b_1^{n-1} =  (n-1)! \, \frac{a_1}{a_2} \ldots \frac{a_1}{a_n} b_1^{n-1},
\]
we can solve for $b_1$ uniquely, since we know that $b_1 > 0$ and all the $a_i > 0$. Knowing $b_1$ clearly gives us all the other $b_i$. This shows that $\omega$ exists locally, while its global existence is a consequence of its uniqueness.
\end{proof}

\theoremstyle{plain}
\newtheorem{cauchy-schw}[positive]{Lemma}
\begin{cauchy-schw} \label{thm:cauchy-schw}
Let $(M, I)$ be a compact complex manifold and suppose that its tangent space $TM$ decomposes into a direct sum $TM = E \oplus F$ of complex subbundles $E$ and $F$. If $\omega, \omega'$ are real (1,1)-forms on $M$ such that $\omega$ is strictly positive when restricted to $E$, while $\omega'$ is strictly positive on $F$ and $E \subseteq \ker \omega'$, there exists a number $T > 0$ such that $\omega + T\omega'$ is strictly positive on $M$.
\end{cauchy-schw}

\begin{proof}
The problem is local in nature by compactness of $M$, since if $\left\{U_i\right\}$ is a cover of $M$ such that $\omega + T_i \omega'$ is strictly positive on $U_i$, then taking a finite subcover and letting $T$ be the maximum of the corresponding $T_i$'s, we get a strictly positive form $\omega + T \omega'$ on the whole $M$.

Let $\omega = \omega_1 + \omega_2 + \omega_3$ be the decomposition of $\omega$ according to the direct sum
\[
\Lambda^2 \left(E^* \oplus F^*\right) = \Lambda^2 (E^*) \oplus \left(E^* \otimes F^*\right) \oplus \Lambda^2 (F^*),
\]
and observe that $\omega'$ lies entirely in the third summand. By assumption of strict positivity, $\omega_1$ is a Hermitian form on $E$, hence comes from a Hermitian metric. Choosing a local orthonormal frame $\left\{e_1, Ie_1, \ldots, e_k, Ie_k\right\}$ for this metric, we can express $\omega_1$ as
\[
\omega_1 = \sum_{i=1}^k e_i \wedge Ie_i,
\]
where we regard the $e_i$ as elements of $E^* \cong E$. Similarly, $\omega'$ is a Hermitian form on $F$ induced by some Hermitian metric. By simple linear algebra, there exists a local orthonormal frame $\left\{f_1, If_1, \ldots, f_l, If_l\right\}$ of $F$ in which the two forms decompose as
\[
\omega_3 = \sum_{j=1}^l a_j \, f_j \wedge If_j, \ \omega' = \sum_{j=1}^l f_j \wedge If_j,
\]
where again we regard $f_j$ as sections of $F^* \cong F$. Clearly, we can choose $T > 0$ such that on some neighborhood, $\omega_3 + T\omega'$ is strictly positive on $F$. This makes $\omega + T \omega'$ locally strictly positive when restricted to both $E$ and $F$, so we only need to take care of the $\omega_2$ term. For this, it is enough to show that we can choose $T$ such that $\omega_1 + \omega_2 + T\omega'$ is locally strictly positive. Let
\[
X = \sum_{i=1}^{2k} \left(X_{2i-1} e_i + X_{2i} Je_i\right), \ Y = \sum_{i=1}^{2l} \left(X_{2i-1} f_i + X_{2i} Jf_i\right)
\]
be arbitrary nonvanishing sections of $E$, $F$ written in the above bases and let $t >0$. We want to show that plugging in $(X + tY, J(X+tY))$ into the above form always gives a strictly positive number:
\[
\omega_1(X, JX) + \omega_2(X, tJY) + \omega_2(tY, JX) + T\omega'(tY, tJY) > 0,
\]
\[
\omega_1(X, JX) + 2t \, \omega_2(X, JY) + t^2 \, T \omega'(Y, JY) > 0.
\]
Thinking of this as a quadratic equation in $t$, its strict positivity is equivalent to the discriminant being negative:
\[
4 \, \omega_2(X, JY)^2 - 4 T \, \omega_1(X, JX) \omega'(Y, JY) < 0,
\]
\[
\omega_2(X, JY)^2 < T \, \omega_1(X, JX) \omega'(Y, JY).
\]
Writing out the right hand side in the bases $\left\{e_i\right\}$, $\left\{f_j\right\}$, we get
\[
T\left(\sum_{i=1}^{2k} X_i^2\right)\left(\sum_{i=1}^{2l} Y_i^2\right),
\]
whereas
\[
\omega_2(X, JY) = \sum_{i,j} c_{ij} X_i Y_j,
\]
for some coefficients $c_{ij}$. Applying the Cauchy-Schwarz inequality to $\omega_2(X, JY)^2$, we get
\[
\left(\sum_{i,j} c_{ij} X_i Y_j\right)^2 \le \sum_{i,j} \left(c_{ij}\right)^2 \sum_{i,j} X_i^2 Y_j^2 = \sum_{i,j} \left(c_{ij}\right)^2 \left(\sum_{i=1}^{2k} X_i^2\right) \left(\sum_{j=1}^{2l} Y_j^2\right).
\]
The sum $\sum_{i,j} \left(c_{ij}\right)^2$ is clearly locally bounded by some $T > 0$, which gives the required inequality.
\end{proof}

\section{Twistor spaces of hyperk\"ahler manifolds} In this section we introduce hypercomplex and hyperk\"ahler manifolds and their twistor spaces. We then present the argument given by Kaledin and Verbitsky in Section 4.4 of \cite{kaled-verbit} that shows that twistor spaces of hyperk\"ahler manifolds are balanced.

A $\emph{hypercomplex}$ manifold $M$ is a $C^\infty$ manifold equipped with three almost complex structures $I, J, K : TM \to TM$ that are integrable and satisfy the quaternionic relations $I^2 = J^2 = K^2 = -1$, $IJ = -JI = K$. Thus, there is an action of the quaternion algebra $\mathbb{H}$ on the tangent bundle $TM$, making each tangent space $T_mM$ into a quaternionic vector space. Observe that each element in
\[
S^2 = \left\{ x_1I + x_2J + x_3K : x_1^2 + x_2^2 + x_3^2  = 1\right\} = \left\{q \in \mathbb{H} : q^2 = -1 \right\} \subseteq \Imag \mathbb{H},
\]
defines an almost complex structure on $TM$, which is integrable since $I$, $J$ and $K$ are. Thus, we have a sphere $S^2$ of complex structures associated to each hypercomplex manifold $M$. A \emph{hyperhermitian metric} on $M$ is a Riemannian metric on $TM$ which is Hermitian with respect to the complex structures $I$, $J$, $K$, and hence also with respect to the whole sphere of complex structures on $M$. Just like the usual Hermitian metrics, these always exist, since starting with an arbitrary Riemannian metric $g_0$, we can define $\forall X, Y \in \Gamma(TM)$,
\[
g(X, Y) := g_0(X, Y) + g_0(IX, IY) + g_0(JX, JY) + g_0(KX, KY),
\]
which is clearly hyperhermitian. For any complex structure $A \in S^2$ on $M$, let
\[
\omega_A(X, Y) := g(AX, Y) \ \ \forall X, Y \in \Gamma(TM)
\]
denote the corresponding Hermitian form. If all the $\omega_A$ are closed, the manifold $M$ is called \emph{hyperk\"ahler}. This is clearly equivalent to $\omega_I$, $\omega_J$ and $\omega_K$ being closed, and in fact it's enough to check only two of them, since we have $\forall A \in S^2$,
\[
d\omega_A = 0 \iff \nabla^{LC} A = 0,
\]
so if $\nabla^{LC} I = \nabla^{LC} J = 0$, for example, then $\nabla^{LC} K = \nabla^{LC} IJ = 0$. In view of this, an equivalent formulation of the hyperk\"ahler condition is that the form
\[
\Omega_I = \omega_J + \sqrt{-1}\omega_K
\]
is closed. This turns out to be a form of type $(2,0)$ with respect to the structure $I$.

For $M$ hypercomplex, we define its \emph{twistor} space $Z$ to be the product manifold $Z = M \times S^2$, where $S^2$ parametrizes the complex structures on $M$, as above. The space $Z$ itself has a natural complex structure, as follows. The sphere $S^2 = \left\{ x_1^2 + x_2^2 + x_3^2  = 1 \right\}$ is identified with the complex projective line $\mathbb{CP}^1$ via stereographic projections
\[
\begin{array}{ccc}
P_N : S^2 \setminus \left\{(0, 0, 1)\right\} & \longleftrightarrow & \mathbb{C} \\
(x_1, x_2, x_3) & \longmapsto & \frac{x_1 - \sqrt{-1}x_2}{1-x_3} \\
\left(\frac{z+\bar{z}}{1+|z|^2}, \frac{\sqrt{-1}(z-\bar{z})}{1 + |z|^2}, \frac{-1+|z|^2}{1+|z|^2}\right) & \longmapsfrom & z
\end{array} \begin{array}{ccc}
P_S : S^2 \setminus \left\{(0, 0, -1)\right\} & \longleftrightarrow & \mathbb{C} \\
(x_1, x_2, x_3) & \longmapsto & \frac{x_1 + \sqrt{-1}x_2}{1+x_3} \\
\left(\frac{w+\bar{w}}{1+|w|^2}, \frac{\sqrt{-1}(\bar{w}-w)}{1 + |w|^2}, \frac{1-|w|^2}{1+|w|^2}\right) & \longmapsfrom & w
\end{array}
\]
Let $I_{\mathbb{CP}^1} : T\mathbb{CP}^1 \to T\mathbb{CP}^1$ denote the almost complex structure on $\mathbb{CP}^1$. Given any point $(m, A) \in M \times \mathbb{CP}^1 = Z$, note that $T_{(m, A)}Z = T_m M \oplus T_A \mathbb{CP}^1$. We define $\mathcal{I} : T_{(m, A)}Z \to T_{(m, A)}Z$ as follows:
\[
\begin{array}{ccccc}
\mathcal{I} & : & T_m M \oplus T_A \mathbb{CP}^1 & \longrightarrow & T_m M \oplus T_A \mathbb{CP}^1. \\
&& (X, V) & \longmapsto & \left(AX, I_{\mathbb{CP}^1}V\right)
\end{array}
\]
It's clear that this defines an almost complex structure on $Z$, and in fact it can be shown that it is integrable, making $Z$ into a complex manifold \cite{kaled} of complex dimension $n+1$, where $\dim_{\mathbb{C}}M = n$. There are canonical projections
\[
\xymatrix{& Z \ar[rd]^\pi \ar[dl]_\sigma & \\ M & & \mathbb{CP}^1,}
\]
the second of which is a holomorphic map, and also, $\forall m \in M$, the canonical section $\mathbb{CP}^1 \to \left\{m\right\} \times \mathbb{CP}^1$ is holomorphic; it is called the \emph{horizontal twistor line} corresponding to $m \in M$. Continuing the analogy, in the direct sum decomposition $T_{(m, A)}Z = T_m M \oplus T_A \mathbb{CP}^1$, vectors in $T_m M$ are called \emph{vertical}, while vectors in  $T_A \mathbb{CP}^1$ are \emph{horizontal}, and similarly for 1-forms. The canonical antiholomorphic involution $\iota : \mathbb{CP}^1 \to \mathbb{CP}^1$, which is just the antipodal map on $S^2$, induces an antiholomorphic map
\[
\iota' = \id \times \iota : M \times \mathbb{CP}^1 \longrightarrow M \times \mathbb{CP}^1
\]
on $Z = M \times \mathbb{CP}^1$, which satisfies $\iota \circ \pi = \pi \circ \iota'$. It can be shown \cite{hklr,verbit} that the hypercomplex manifold $M$ can be recovered from the horizontal twistor lines in $Z$ which are completely described by three pieces of information, namely the complex structure on $Z$, the holomorphic projection $\pi : Z \to \mathbb{CP}^1$, and the antiholomorphic involution $\iota'$ that satisfies $\iota \circ \pi = \pi \circ \iota'$.

In case $M$ is hyperk\"ahler, there is one extra structure on the twistor space $Z$, which comes from the K\"ahler metric on $M$. Let $g_M$ denote this metric, and $g_{\mathbb{CP}^1}$ the usual Fubini-Study metric on $\mathbb{CP}^1$. Then
\[
g := \sigma^*\left(g_M\right) + \pi^*\left(g_{\mathbb{CP}^1}\right)
\]
is easily verified to be a Hermitian metric on $Z$. To simplify notation, we write $g = g_M + g_{\mathbb{CP}^1}$. At a point $(m, A) \in Z$, the corresponding Hermitian form $\omega$ decomposes as follows:
\[
\omega\left((X,V), (X',V')\right) = \omega_M(X,X') + \omega_{\mathbb{CP}^1}(V,V') = g_M(AX, X') + g_{\mathbb{CP}^1}(I_{\mathbb{CP}^1}V, V'),
\]
where $(X, V), (X', V') \in T_m M \oplus T_A \mathbb{CP}^1 = T_{(m, A)} Z$. By analogy with the case of twistor spaces of self-dual 4-manifolds, $\omega$ need not be closed; in fact, if $M$ is compact, $Z$ can never be K\"ahler, as we will show in the next section. However, the Hermitian metric on $Z$ is always balanced when $M$ is hyperk\"ahler, as we show next.

\theoremstyle{plain}
\newtheorem{hyperk-balanced}{Theorem}
\begin{hyperk-balanced} \label{thm:hyperk-balanced} (Kaledin-Verbitsky)
Let $(M, I, J, K, g_M)$ be a hyperk\"ahler manifold of complex dimension $n$. Then its twistor space $Z$ with the Hermitian metric induced from the hyperk\"ahler structure is balanced.
\end{hyperk-balanced}

\begin{proof} We closely follow the argument laid out in Section 4.4 of \cite{kaled-verbit}.
In the notation used above, we need to show that $d\left(\omega^n\right) = 0$. This is clearly equivalent to showing
\[
\omega^{n-1} \wedge d\omega = 0.
\]
Observe that we have a decomposition of the differential operator $d = d_M + d_{\mathbb{CP}^1}$ according to the direct sum $TZ = TM \oplus T\mathbb{CP}^1$. Since $\omega = \omega_M + \omega_{\mathbb{CP}^1}$, we have
\[
d\omega = d_M \omega_M + d_{\mathbb{CP}^1} \omega_M + d_M \omega_{\mathbb{CP}^1} + d_{\mathbb{CP}^1} \omega_{\mathbb{CP}^1}.
\]
The first term is zero by the hyperk\"ahler condition on $M$, while the last two terms are zero because $\omega_{\mathbb{CP}^1}$ is a pullback of a closed form on $\mathbb{CP}^1$ to $Z$. We need to investigate the second term. To simplify our argument, we will work over a fixed horizontal twistor line $\left\{m\right\} \times \mathbb{CP}^1 \subseteq Z$.

Let
\[
W := \Imag \mathbb{H} = \left\{ aI + bJ + cK \right\} \cong \mathbb{R}^3,
\]
and let $\mathcal{W} = \mathbb{CP}^1 \times W$ be the corresponding trivial bundle. When we view $\mathbb{CP}^1$ as a parametrization of the complex structures on $M$, it's just the unit sphere $S^2 \subseteq W$, hence we can view $\mathcal{W}$ as the restriction $\mathcal{W} = \left. TW\right|_{S^2}$. There is a canonical embedding of $\mathcal{W}$ into the (trivial) bundle of vertical 2-forms over the horizontal line $\left\{m\right\} \times \mathbb{CP}^1$:
\[
\begin{array}{rcl}
\mathcal{W} = \mathbb{CP}^1 \times W & \longrightarrow & \left\{m\right\} \times \mathbb{CP}^1 \times \Lambda^2_m M \\
\left(A, aI + bJ + cK\right) & \longmapsto & \left(m, A, a\omega_I + b\omega_J + c\omega_K\right).
\end{array}
\]
Given an element $V = aI + bJ +cK$ of $W$, we denote by $\omega_V = a\omega_I + b\omega_J + c\omega_K$ its image under this mapping. In this way, we can think of $\mathcal{W}$ as a bundle of vertical 2-forms over $\left\{m\right\} \times \mathbb{CP}^1$ with global frame $\left\{\omega_I, \omega_J , \omega_K\right\}$. Since $d_{\mathbb{CP}^1} \omega_I = d_{\mathbb{CP}^1} \omega_J = d_{\mathbb{CP}^1} \omega_K = 0$, we can think of the operator $d_{\mathbb{CP}^1}$ on $\mathcal{W}$ as a flat connection
\[
\begin{array}{rcl}
d_{\mathbb{CP}^1} = \nabla : \Gamma\left(\mathcal{W}\right) & \longrightarrow & \Gamma\left(\Lambda^1 \mathbb{CP}^1 \otimes \mathcal{W}\right) \\
f_1 \omega_I + f_2 \omega_J + f_3 \omega_K & \longmapsto & df_1 \otimes \omega_I + df_2 \otimes \omega_J + df_3 \otimes \omega_K.
\end{array}
\]
Of course, this is just the usual Euclidean connection on $\mathbb{R}^3 \cong \Imag \mathbb{H}$ restricted to $S^2 \cong \mathbb{CP}^1$. Note that $\mathcal{W} = \left.TW\right|_{S^2} \cong \left.T\mathbb{R}^3\right|_{S^2} = N \oplus TS^2$, where $N$ is the normal bundle of the embedding $S^2 \subseteq \mathbb{R}^3$ and $TS^2$ is the tangent bundle. At the point $A = (a_1, a_2, a_3) \in S^2 \cong \mathbb{CP}^1$, we have
\[
N_A = \left\{ \lambda \, a_1 \omega_I + \lambda \, a_2\omega_J + \lambda \, a_3\omega_K : \lambda \in \mathbb{R} \right\},
\]
\[
T_AS^2 = \left\{v_1 \, \omega_I + v_2 \, \omega_J + v_3 \, \omega_K : a_1 v_1 + a_2 v_2 + a_3 v_3 = 0 \right\}.
\]
Thus, $N$ is a trivial bundle with a global trivialization given by $\omega_M = x_1\omega_I + x_2\omega_J + x_3\omega_K$, while the almost complex structure $I_{\mathbb{CP}^1} : T\mathbb{CP}^1 \to T\mathbb{CP}^1$ at the point $A \in \mathbb{CP}^1$ is given by the quaternion multiplication $V \mapsto AV$, where we once again think of $A \in N_A$, $V \in T_AS^2$ as elements of $W$. We want to compute
\[
d_{\mathbb{CP}^1} \omega_M = \nabla \left(x_1 \omega_I + x_2 \omega_J + x_3 \omega_K\right) = dx_1 \otimes \omega_I + dx_2 \otimes \omega_J + dx_3 \otimes \omega_K.
\]
Fix a point $A = (a_1, a_2, a_3)$ in $\left\{m\right\} \times \mathbb{CP}^1 $ and look at the decomposition $d_{\mathbb{CP}^1}\omega_M = \partial_{\mathbb{CP}^1} \omega_M + \bar{\partial}_{\mathbb{CP}^1} \omega_M$. We claim that $\bar{\partial}_{\mathbb{CP}^1} \omega_M \in \Gamma\left(\Lambda^{0,1}\mathbb{CP}^1 \otimes \Lambda_m^{2,0} M\right)$, where the complex structure on $T_mM$ is understood to be $A$. To verify this, we use the description of $d_{\mathbb{CP}^1}$ as the connection $\nabla$ and plug in an arbitrary vector $V + \sqrt{-1} I_{\mathbb{CP}^1}V = V + \sqrt{-1} AV \in T_A^{0,1} \mathbb{CP}^1$, where $V = (v_1, v_2, v_3) \in T_A \mathbb{CP}^1$ is real.
\[
\nabla_{V + \sqrt{-1} AV} \omega_M = v_1 \omega_I + v_2 \omega_J + v_3 \omega_K +
\]
\[
+ \sqrt{-1}(a_2v_3 - a_3 v_2)\omega_I + \sqrt{-1}(a_3v_1 - a_1v_3)\omega_J + \sqrt{-1}(a_1v_2 - a_2v_1)\omega_K =
\]
\[
= \omega_V + \sqrt{-1} \omega_{AV}.
\]
Plugging into this form an arbitrary vector $X \in T_mM$ and a (0,1)-vector $Y \in T_m^{0,1}M$ (with respect to the complex structure $A$), we get
\[
\omega_V(X, Y) + \sqrt{-1}\omega_{AV}(X, Y) = g(VX, Y) + \sqrt{-1}g(AVX, Y) =
\]
\[
= g(VX, Y) + \sqrt{-1}g(A(AV)X, AY) = g(VX, Y) + \sqrt{-1}g(-VX, -\sqrt{-1}Y) = 0.
\]
Hence $\bar{\partial}_{\mathbb{CP}^1} \omega_M \in \Gamma\left(\Lambda^{0,1}\mathbb{CP}^1 \otimes \Lambda_m^{2,0} M\right)$ and $\partial_{\mathbb{CP}^1} \omega_M \in \Gamma\left(\Lambda^{1,0}\mathbb{CP}^1 \otimes \Lambda_m^{0,2} M\right)$, since $\omega_M$ is real and $\bar{\partial}_{\mathbb{CP}^1} \omega_M$ is the conjugate of $\partial_{\mathbb{CP}^1} \omega_M$.

We now examine the form $\omega^{n-1} \wedge d\omega$. We have
\[
\omega^{n-1} \wedge d\omega = \left(\omega_M + \omega_{\mathbb{CP}^1}\right)^{n-1} \wedge d_{\mathbb{CP}^1} \omega_M  = \omega_M^{n-1} \wedge \partial_{\mathbb{CP}^1} \omega_M + \omega_M^{n-1} \wedge \bar{\partial}_{\mathbb{CP}^1} \omega_M +
\]
\[
+ (n-1) \omega_M^{n-2}\wedge \omega_{\mathbb{CP}^1} \wedge \partial_{\mathbb{CP}^1} \omega_M + (n-1) \omega_M^{n-2}\wedge \omega_{\mathbb{CP}^1} \wedge \bar{\partial}_{\mathbb{CP}^1} \omega_M.
\]
Since $\omega_M^{n-1} \in \Lambda_m^{n-1,n-1}M$, the vertical bidegree of the first two terms is $(n-1, n+1)$, $(n+1, n-1)$, respectively, making them zero, since $\dim_{\mathbb{C}}M = n$. On the other hand, the degree of the horizontal part of the last two terms is $3 > 2 = \dim_{\mathbb{R}} \mathbb{CP}^1$, making them zero as well.

\end{proof}

\section{Twistor spaces of hypercomplex manifolds} \label{sec:hypercomplex}
We now come to the main part of our exposition: a generalization of Theorem \ref{thm:hyperk-balanced} for compact hyperhermitian (and hence general hypercomplex) manifolds $M$. In contrast to the hyperk\"ahler case, the product metric on $Z = M \times \mathbb{CP}^1$ need not be balanced. Instead of looking at the form
\[
\omega^n =  \left(\omega_M + \omega_{\mathbb{CP}^1}\right)^n = \omega_M^n + n \omega_M^{n-1} \wedge \omega_{\mathbb{CP}^1},
\]
we will instead use Lemma \ref{thm:cauchy-schw} to show that a certain linear combination of forms
\[
\alpha \, \omega_M^n + \beta \, dd^c \left(\omega_M^{n-1}\right)
\]
is a closed strictly positive $(n,n)$-form on $Z = M \times \mathbb{CP}^1$, and then use Lemma \ref{thm:positive} to conclude that $Z$ is balanced. We will need compactness of $M$ in order to apply Lemma \ref{thm:cauchy-schw}.

\theoremstyle{plain}
\newtheorem{hyperc-balanced}[hyperk-balanced]{Theorem}
\begin{hyperc-balanced}
Let $(M, I, J, K, g_M)$ be a compact hyperhermitian manifold of complex dimension $n$. Then its twistor space $Z$ is balanced.
\end{hyperc-balanced}

\begin{proof}
The volume form on $Z = M \times \mathbb{CP}^1$ with respect to the product metric is given by
\[
\Omega_Z = \frac{\omega^{n+1}}{(n+1)!} = \frac{\left(\omega_M + \omega_{\mathbb{CP}^1}\right)^{n+1}}{(n+1)!} = \frac{(n+1)\omega_M^n \wedge \omega_{\mathbb{CP}^1}}{(n+1)!} = \Omega_M \wedge \omega_{\mathbb{CP}^1},
\]
where $\Omega_M$ denotes the pullback of the volume form on $M$ via the projection $\sigma : Z \to M$. Note that $\omega_M^n = n! \, \Omega_M$ and $dd^c\left(\omega_M^{n-1}\right)$ are closed $(n,n)$-forms on $Z$ and we can think of them as elements of $\Lambda^{1,1} TZ$ via the isomorphism induced by the volume form $\Omega_Z$. Because the metric on $Z$ induces an isomorphism $TZ \cong \Lambda^1 Z$, we will be able to apply Lemma \ref{thm:cauchy-schw} if we can show that $\omega_M^n$ is strictly positive and vertical forms lie in its kernel, while $\pm dd^c\left(\omega_M^{n-1}\right)$ (the sign will depend on the dimension of $M$) is strictly positive when restricted to vertical forms in $\Lambda^1 M$. The first statement is easy, since $\omega_M^n$ is a constant multiple of the vertical volume form $\Omega_M$, and we know that $\Omega_Z = \Omega_M \wedge \omega_{\mathbb{CP}^1}$. For the second statement, since we only need to establish strict positivity on vertical forms, it's enough to restrict to a horizontal twistor line $\left\{m\right\} \times \mathbb{CP}^1$ and consider the form
\[
d_{\mathbb{CP}^1}d_{\mathbb{CP}^1}^c\left(\omega_M^{n-1}\right) = 2\sqrt{-1}(n-1)\left(\partial_{\mathbb{CP}^1}\bar{\partial}_{\mathbb{CP}^1} \omega_M\right) \wedge \omega_M^{n-2} +
\]
\[
+ 2\sqrt{-1}(n-1)(n-2)\left(\partial_{\mathbb{CP}^1} \omega_M\right) \wedge \left(\bar{\partial}_{\mathbb{CP}^1} \omega_M\right) \wedge \omega_M^{n-3}.
\]
Note that if $n = 2$, the second term vanishes. We will now use our description of the $d_{\mathbb{CP}^1}$ operator as a connection from the proof of Theorem \ref{thm:hyperk-balanced} to show that both of these terms are multiples of $\omega_{\mathbb{CP}^1} \wedge \omega_M^{n-1}$, which is strictly positive since products of positive forms are positive (see \cite{demailly}, Section III.1). We will compute the forms
\[
\partial_{\mathbb{CP}^1}\bar{\partial}_{\mathbb{CP}^1} \omega_M = \nabla^{1,0} \nabla^{0,1} \omega_M = \partial \bar{\partial} x_1 \otimes \omega_I + \partial \bar{\partial} x_2 \otimes \omega_J + \partial \bar{\partial} x_3 \otimes \omega_K,
\]
\[
\partial_{\mathbb{CP}^1} \omega_M = \nabla^{1,0} \omega_M = \partial x_1 \otimes \omega_I + \partial x_2 \otimes \omega_J + \partial x_3 \otimes \omega_K,
\]
\[
\bar{\partial}_{\mathbb{CP}^1} \omega_M = \nabla^{0,1} \omega_M = \bar{\partial} x_1 \otimes \omega_I + \bar{\partial} x_2 \otimes \omega_J + \bar{\partial} x_3 \otimes \omega_K
\]
in the local holomorphic coordinate $z = x + \sqrt{-1}y$ on $S^2 \cong \mathbb{CP}^1$ coming from the stereographic projection $P_N$ from the point $(1,0,0)$ (see previous section). The computation in the other chart is completely analogous, and we will omit it. The Fubini-Study metric $\omega_{\mathbb{CP}^1}$ takes the form
\[
\omega_{\mathbb{CP}^1} = \sqrt{-1}\partial\bar{\partial} \log \left(1+ |z|^2\right) = \sqrt{-1}\partial \left(\frac{z \, d\bar{z}}{1+|z|^2}\right) = \frac{\sqrt{-1} dz \wedge d\bar{z}}{\left(1+|z|^2\right)^2}.
\]
Calculating the various partial derivatives of $x_1$, $x_2$, $x_3$, we get
\[
\begin{array}{cc}
\partial x_1 = \partial\left(\frac{z+\bar{z}}{1+|z|^2}\right) = \frac{1-\bar{z}^2}{\left(1+|z|^2\right)^2} \, dz & \bar{\partial} x_1 = \bar{\partial}\left(\frac{z+\bar{z}}{1+|z|^2}\right) = \frac{1-z^2}{\left(1+|z|^2\right)^2} \, d\bar{z} \\
\partial x_2 = \partial\left(\frac{\sqrt{-1}(z-\bar{z})}{1 + |z|^2}\right) = \frac{\sqrt{-1}\left(1 +\bar{z}^2\right)}{\left(1+|z|^2\right)^2} \, dz & \bar{\partial} x_2 = \bar{\partial}\left(\frac{\sqrt{-1}(z-\bar{z})}{1 + |z|^2}\right) = \frac{-\sqrt{-1}\left(1 +z^2\right)}{\left(1+|z|^2\right)^2} \, d\bar{z} \\
\partial x_3 = \partial\left(\frac{-1+|z|^2}{1+|z|^2}\right) = \frac{2\bar{z}}{\left(1+|z|^2\right)^2} \, dz & \bar{\partial} x_3 = \bar{\partial}\left(\frac{-1+|z|^2}{1+|z|^2}\right) = \frac{2z}{\left(1+|z|^2\right)^2} \, d\bar{z} \\
\end{array}
\]
\[
\partial\bar{\partial}x_1 = \frac{-2(z+\bar{z}) \, dz \wedge d\bar{z}}{\left(1+|z|^2\right)^3}, \ \partial\bar{\partial}x_2 = \frac{-2\sqrt{-1}(z-\bar{z}) \, dz \wedge d\bar{z}}{\left(1+|z|^2\right)^3}, \ \partial\bar{\partial}x_3 = \frac{-2\left(-1+|z|^2\right) \, dz \wedge d\bar{z}}{\left(1+|z|^2\right)^3}.
\]
Thus,
\[
\sqrt{-1} \partial_{\mathbb{CP}^1}\bar{\partial}_{\mathbb{CP}^1} \omega_M = -2\left(\frac{\sqrt{-1} dz \wedge d\bar{z}}{\left(1+|z|^2\right)^2}\right) \otimes \left(\frac{z+\bar{z}}{1+|z|^2} \, \omega_I + \frac{\sqrt{-1}(z-\bar{z})}{1 + |z|^2} \, \omega_J + \frac{-1+|z|^2}{1+|z|^2} \, \omega_K\right) =
\]
\[
= -2 \, \omega_{\mathbb{CP}^1} \wedge \omega_M,
\]
from which we conclude that
\[
2\sqrt{-1}(n-1)\left(\partial_{\mathbb{CP}^1}\bar{\partial}_{\mathbb{CP}^1} \omega_M\right) \wedge \omega_M^{n-2} = - 4 (n-1) \omega_{\mathbb{CP}^1} \wedge \omega_M^{n-1}.
\]
If $n=2$, then, as we noted above, this is equal to $d_{\mathbb{CP}^1}d^c_{\mathbb{CP}^1} \left(\omega_M^{n-1}\right)$, so taking the negative of $dd^c\left(\omega_M^{n-1}\right)$ gives a form that is strictly positive on vertical 1-forms, and we can apply Lemma \ref{thm:cauchy-schw} to conclude that $\exists \, T > 0$ such that
\[
T \, \omega_M^n - dd^c\left(\omega_M^{n-1}\right)
\]
is strictly positive. For the case $n > 2$, we also need to examine the other term. We know that at any point $A \in \mathbb{CP}^1$, for any $V \in T_A\mathbb{CP}^1$,
\[
\partial_{\mathbb{CP}^1} \omega_M (V - \sqrt{-1} AV) = \omega_V - \sqrt{-1} \omega_{AV},
\]
\[
\partial_{\mathbb{CP}^1} \omega_M (V + \sqrt{-1} AV) = \omega_V + \sqrt{-1} \omega_{AV}.
\]
If we take $V = \frac{1}{2}\frac{\partial}{\partial x}$, then $V - \sqrt{-1} AV = \frac{\partial}{\partial z}$, $V + \sqrt{-1} AV = \frac{\partial}{\partial \bar{z}}$, and we conclude from the above that
\[
\partial_{\mathbb{CP}^1} \omega_M = dz \wedge \omega_{\frac{\partial}{\partial z}} = dz \wedge \left(\omega_{V} - \sqrt{-1} \omega_{AV}\right),
\]
\[
\bar{\partial}_{\mathbb{CP}^1} \omega_M = d\bar{z} \wedge \omega_{\frac{\partial}{\partial \bar{z}}} = d\bar{z} \wedge \left(\omega_V + \sqrt{-1} \omega_{AV}\right).
\]
Hence
\[
\sqrt{-1}\left(\partial_{\mathbb{CP}^1} \omega_M\right) \wedge \left(\bar{\partial}_{\mathbb{CP}^1} \omega_M\right) = \sqrt{-1} \, dz \wedge \omega_{\frac{\partial}{\partial z}} \wedge d\bar{z} \wedge \omega_{\frac{\partial}{\partial \bar{z}}} =
\]
\[
= \frac{\sqrt{-1} dz \wedge d\bar{z}}{\left(1+|z|^2\right)^2} \wedge \left(1+|z|^2\right) \omega_{\frac{\partial}{\partial z}} \wedge \left(1+|z|^2\right) \omega_{\frac{\partial}{\partial z}} = \omega_{\mathbb{CP}^1} \wedge \Psi \wedge \bar{\Psi}.
\]
We now compute the expression $\Psi \wedge \bar{\Psi} \wedge \omega_M^{n-3}$. To simplify things we only do the computation at the point $z = 1$, which corresponds to $I \in \mathbb{CP}^1$, where it takes the form
\[
\left(\omega_K + \sqrt{-1}\omega_J\right) \wedge \left(\omega_K - \sqrt{-1}\omega_J\right) \wedge \omega_I^{n-3} = \left(\omega_J + \sqrt{-1}\omega_K\right) \wedge \left(\omega_J - \sqrt{-1}\omega_K\right) \wedge \omega_I^{n-3},
\]
while at a general point $A \in \mathbb{CP}^1$ corresponding to $z \in \mathbb{C}$, an entirely analogous argument applies, except that $(I, J, K)$ need to be replaced by $\left(A, \frac{1+|z|^2}{2}\frac{\partial}{\partial x}, \frac{1+|z|^2}{2}\frac{\partial}{\partial y}\right)$, which form a quaternionic triple in the space $W = N_A \oplus T_A \mathbb{CP}^1$.

The vertical tangent space $T_mM$ to the point $(m, I) \in M \times \mathbb{CP}^1$ is a quaternionic vector space with respect to the triple $(I, J, K)$, so we can identify it with $\mathbb{H}^k$, where $k = \dim_\mathbb{H} M = \frac{1}{2}\dim_\mathbb{C} M = \frac{n}{2}$. The metric $g$ restricted to $T_mM$ is quaternionic-hermitian, hence we can find a quaternionic orthonormal basis $\left\{e_1, \ldots, e_k\right\}$ of $T_m M$; let $\left\{e_1^*, \ldots, e_k^*\right\}$ denote the dual basis of $\Lambda_m^1 M$. We define the following complex-valued 1-forms $\forall 1 \le i \le k$, which constitute a complex basis of $\Lambda_m^1 M \otimes \mathbb{C} = \Lambda_m^{1,0} M \oplus \Lambda_m^{0,1} M$, where the decomposition is relative to the complex structure $I$.
\[
\begin{array}{cc}
d\zeta := e_i^* + \sqrt{-1} Ie_i^* & d\xi = Je_i^* + \sqrt{-1}Ke_i^* \\[.1cm]
d\bar{\zeta} := e_i^* - \sqrt{-1} Ie_i^* & d\bar{\xi} = Je_i^* - \sqrt{-1}Ke_i^*
\end{array}
\]
With respect to this basis, it's not hard to see that the forms $\omega_I$, $\omega_J$, $\omega_K$ decompose as follows:
\[
\omega_I = \sum_{i=1}^k\left(\frac{\sqrt{-1}}{2} \, d\zeta_i \wedge d\bar{\zeta}_i + \frac{\sqrt{-1}}{2} \, d\xi_i \wedge d\bar{\xi}_i \right),
\]
\[
\omega_J = \sum_{i=1}^k\left(\frac{1}{2} \, d\zeta_i \wedge d\xi_i + \frac{1}{2} \, d\bar{\zeta}_i \wedge d\bar{\xi}_i \right),
\]
\[
\omega_K = \sum_{i=1}^k\left(-\frac{\sqrt{-1}}{2} \, d\zeta_i \wedge d\xi_i + \frac{\sqrt{-1}}{2} \, d\bar{\zeta}_i \wedge d\bar{\xi}_i \right).
\]
Further computing,
\[
\omega_J + \sqrt{-1}\omega_K = \sum_{i=1}^k d\zeta_i \wedge d\xi_i,
\]
hence
\[
\left(\omega_J + \sqrt{-1}\omega_K\right) \wedge \left(\omega_J - \sqrt{-1}\omega_K\right) \wedge \omega_I^{n-3} =
\]
\[
= \left(\sum_{i=1}^k d\zeta_i \wedge d\xi_i,\right) \wedge \left(\sum_{i=1}^k d\bar{\zeta}_i \wedge d\bar{\xi}_i,\right) \wedge \left\{\sum_{i=1}^k\left(\frac{\sqrt{-1}}{2} \, d\zeta_i \wedge d\bar{\zeta}_i + \frac{\sqrt{-1}}{2} \, d\xi_i \wedge d\bar{\xi}_i \right)\right\}^{n-3} =
\]
\[
= -\left(\frac{\sqrt{-1}}{2}\right)^{n-3}(n-1)(n-3)! \sum_{i=1}^k \left(\bigwedge_{j = 1}^k d\zeta_j \wedge d\bar{\zeta}_j\right) \wedge \left(\bigwedge_{j \ne i} d\xi_j \wedge d\bar{\xi}_j\right) -
\]
\[
-\left(\frac{\sqrt{-1}}{2}\right)^{n-3}(n-1)(n-3)! \sum_{i=1}^k \left(\bigwedge_{j \ne i} d\zeta_j \wedge d\bar{\zeta}_j\right) \wedge \left(\bigwedge_{j =1}^k d\xi_j \wedge d\bar{\xi}_j\right).
\]
On the other hand,
\[
\omega_I^{n-1} = \left\{\sum_{i=1}^k\left(\frac{\sqrt{-1}}{2} \, d\zeta_i \wedge d\bar{\zeta}_i + \frac{\sqrt{-1}}{2} \, d\xi_i \wedge d\bar{\xi}_i \right)\right\}^{n-1} =
\]
\[
= \left(\frac{\sqrt{-1}}{2}\right)^{n-1}(n-1)! \sum_{i=1}^k \left(\bigwedge_{j = 1}^k d\zeta_j \wedge d\bar{\zeta}_j\right) \wedge \left(\bigwedge_{j \ne i} d\xi_j \wedge d\bar{\xi}_j\right) +
\]
\[
+ \left(\frac{\sqrt{-1}}{2}\right)^{n-1}(n-1)! \sum_{i=1}^k \left(\bigwedge_{j \ne i} d\zeta_j \wedge d\bar{\zeta}_j\right) \wedge \left(\bigwedge_{j =1}^k d\xi_j \wedge d\bar{\xi}_j\right).
\]
We conclude that
\[
\left(\omega_J + \sqrt{-1}\omega_K\right) \wedge \left(\omega_J - \sqrt{-1}\omega_K\right) \wedge \omega_I^{n-3} = \frac{4}{n-2} \omega_I^{n-1}
\]
at the point $z = 1$, and generally, $\Psi \wedge \bar{\Psi} \wedge \omega_m^{n-3} = \frac{4}{n-2} \omega_M^{n-1}$. We thus have
\[
2\sqrt{-1}(n-1)(n-2)\left(\partial_{\mathbb{CP}^1} \omega_M\right) \wedge \left(\bar{\partial}_{\mathbb{CP}^1} \omega_M\right) \wedge \omega_M^{n-3} =
\]
\[
= 2(n-1)(n-2) \, \omega_{\mathbb{CP}^1} \wedge \Psi \wedge \bar{\Psi} \wedge \omega_M^{n-3} =
\]
\[
= 8(n-1) \, \omega_{\mathbb{CP}^1} \wedge \omega_M^{n-1},
\]
and so if $n > 2$,
\[
d_{\mathbb{CP}^1} d_{\mathbb{CP}^1}^c \left(\omega_M^{n-1}\right) = 4(n-1) \, \omega_{\mathbb{CP}^1} \wedge \omega_M^{n-1},
\]
which is strictly positive on vertical forms, hence applying Lemma 2, we get a $T > 0$ such that
\[
T \, \omega_M^n + dd^c\left(\omega_M^{n-1}\right)
\]
is strictly positive.

Thus both in case $n = 2$ and $n > 2$, we are assured of the existence of a closed strictly positive $(n, n)$-form on $Z$, which immediately implies that $Z$ is balanced by Lemma \ref{thm:positive}. We are finished.
\end{proof}

We conclude with a short corollary demonstrating that the K\"ahler condition is too strong for twistor spaces, as opposed to balancedness.

\theoremstyle{plain}
\newtheorem{hyperc-nonkahler}{Corollary}
\begin{hyperc-nonkahler}
Let $(M, I, J, K, g_M)$ be a compact hyperk\"ahler manifold of complex dimension $n$. Then its twistor space $Z$ is never K\"ahler.
\end{hyperc-nonkahler}

\begin{proof}
In the notations of the previous proof, we have $d _M \omega_M = d_M^c \omega_M = 0$ by the hyperk\"ahler condition on $M$, hence
\[
dd^c \omega_M = \sqrt{-1}\partial_{\mathbb{CP}^1} \bar{\partial}_{\mathbb{CP}^1} \omega_M = -2 \, \omega_{\mathbb{CP}^1} \wedge \omega_M, 
\]
hence if $\omega_Z$ is any K\"ahler form on $Z$,
\[
-dd^c \omega_M \wedge \left(\omega_Z\right)^{n-1}
\]
is an exact strictly positive $(n+1, n+1)$-form on $Z$, in the sense that it is a strict positive multiple of the volume form $\Omega_Z$. But this is impossible, since then
\[
-\int_Z d\left(d^c\omega_M \wedge \left(\omega_Z\right)^{n-1} \right) = -\int_Z dd^c \omega_M \wedge \left(\omega_Z\right)^{n-1} > 0
\]
by compactness, whereas the first integral is zero by Stokes' theorem.
\end{proof}


\begin{thebibliography}{99}

\bibitem[Be]{berger} M. Berger, ``Sur les groupes d'holonomie des vari\'et\'es \`a connexion affine et des vari\'et\'es riemanniennes'', \emph{Bull. Soc. Math. France} \textbf{83} (1955), pp. 279-330

\bibitem[Bo]{boyer} C. P. Boyer, ``A note on hyper-Hermitian four-manifolds'', \emph{Proc. Amer.
Math. Soc.} \textbf{102} (1988), pp. 157-164

\bibitem[C]{calabi} E. Calabi, ``M\'etriques k\"ahl\'eriennes et fibr\'es holomorphes'', \emph{Ann. Sci. \'Ecol. Norm. Sup\'er.} \textbf{12} (1979), pp. 269-294

\bibitem[D]{demailly} J.-P. Demailly, \emph{Complex Analytic and Differential Geometry}, available at http://www-fourier.ujf-grenoble.fr/$\sim$demailly/manuscripts/agbook.pdf

\bibitem[Don1]{donaldson} S. K. Donaldson, ``A new proof of a theorem of Narasimhan and Seshadri'', \emph{J. Differential Geom.} \textbf{18} (1983), pp. 269-278

\bibitem[Don2]{donaldson2} S. K. Donaldson, ``Anti self-dual Yang-Mills connections over complex algebraic surfaces and stable vector bundles'', \emph{Proc. Lond. Math. Soc.} \textbf{50} (1985), pp. 1-26

\bibitem[Don3]{donaldson3} S. K. Donaldson, ``Infinite determinants, stable bundles and curvature'', \emph{Duke Math. J.} \textbf{54} (1987), pp. 231-247

\bibitem[GH]{gray-hervella} A. Gray, L. M. Hervella, ``The sixteen classes of almost Hermitian manifolds and their linear invariants'', \emph{Ann. Mat. Pura Appl.} \textbf{123} (1980), pp. 35Ð58

\bibitem[H]{hitchin} N. J. Hitchin, ``K\"ahlerian twistor spaces'', \emph{Proc. Lond. Math. Soc.} \textbf{43} (1981), pp. 133-150

\bibitem[HKLR]{hklr} N. J. Hitchin, A. Karlhede, U. Lindstr\"om, M. Ro\v cek, ``Hyperk\"ahler metrics and supersymmetry'', \emph{Comm. Math. Phys.} \textbf{108} (1987), pp. 535-589

\bibitem[HP]{howe-pap} P. S. Howe, G. Papadopoulos, ``Twistor spaces for hyper-K\"ahler manifolds with torsion'', \emph{Phys. Lett. B} \textbf{379} (1996), no. 1-4, pp. 80-86

\bibitem[K]{kaled} D. Kaledin, ``Integrability of the twistor space for a hypercomplex manifold'', \emph{Selecta Math. (N.S.)} \textbf{4} (1998), pp. 271Ð278.

\bibitem[KV]{kaled-verbit} D. Kaledin, M. Verbitsky, ``Non-Hermitian Yang-Mills connections'', \emph{Selecta Math. (N.S.)} \textbf{4} (1998), pp. 279-320

\bibitem[Kob]{kobayashi} S. Kobayashi, ``First Chern class and holomorphic tensor fields'', \emph{Nagoya Math. J.} \textbf{77} (1980), pp. 5-11

\bibitem[LY]{li-yau} J. Li, S. T. Yau, ``Hermitian Yang-Mills connections on non-K\"ahler manifolds'', \emph{Mathematical aspects of string theory}, World Scientific Publ., London (1987), pp. 560-573

\bibitem[LT]{lubke-teleman} M. L\"ubke, A. Teleman, \emph{The Kobayashi-Hitchin correspondence}, World Scientific Publ., River Edge, NJ (1995)

\bibitem[Mi]{michel} M. L. Michelsohn, ``On the existence of special metrics in complex geometry'', \emph{Acta Math.} \textbf{149} (1982), no. 3-4, pp. 261-295

\bibitem[Mu]{mumford} D. Mumford, ``Projective invariants of projective structures and applications'', \emph{Proc. Internat. Congr. Mathematicians (Stockh., 1962)}, Inst. Mittag-Leffler, Djursholm (1963), pp. 526-530

\bibitem[NN]{newlander} A. Newlander, L. Nirenberg, ``Complex analytic coordinates in almost-complex manifolds'', \emph{Ann. Math.} \textbf{65} (1957), pp. 391-404.

\bibitem[UY1]{uhlenbeck-yau} K. K. Uhlenbeck, S. T. Yau, ``On the existence of Hermitian Yang-Mills connections in stable vector bundles'', \emph{Comm. Pure Appl. Math.} \textbf{39} (1986), pp. S257-S293

\bibitem[UY2]{uhlenbeck-yau2} K. K. Uhlenbeck, S. T. Yau, ``A note on our previous paper: On the existence of Hermitian Yang-Mills connections in stable vector bundles'', \emph{Comm. Pure Appl. Math.} \textbf{XLII} (1989), pp. 703-707

\bibitem[V1]{verbit} M. Verbitsky, ``Hypercomplex Varieties'', \emph{Comm. Anal. Geom.} \textbf{7} (1999), no. 2, pp. 355-396

\bibitem[V2]{verbit2} M. Verbitsky, ``Balanced HKT metrics and strong HKT metrics on hypercomplex manifolds'', \emph{Math. Res. Lett.} \textbf{16} (2009), no. 4, pp. 735-752

\bibitem[Y]{yau} S. T. Yau, ``On the Ricci curvature of a compact K\"ahler manifold and the complex Monge-Amp\`ere equation I'', \emph{Comm. Pure Appl. Math.} \textbf{31} (1978), pp. 339-411

\end{thebibliography}
\end{document}